\documentclass[12pt,a4paper]{amsart}
\usepackage{amsfonts}
\usepackage{amsthm}
\usepackage{amsmath}
\usepackage{amscd}
\usepackage[latin2]{inputenc}
\usepackage{t1enc}
\usepackage[mathscr]{eucal}
\usepackage{indentfirst}
\usepackage{graphicx}
\usepackage{graphics}
\usepackage{pict2e}
\usepackage{epic}
\numberwithin{equation}{section}
\usepackage[margin=3.0cm]{geometry}
\usepackage{epstopdf}  
\usepackage{bbold}

\theoremstyle{plain}
\newtheorem{Th}{Theorem}[section]
\newtheorem{Lemma}[Th]{Lemma}
\newtheorem{Cor}[Th]{Corollary}
\newtheorem{Prop}[Th]{Proposition}

 \theoremstyle{definition}
\newtheorem{Def}[Th]{Definition}

\newtheorem{Rem}[Th]{Remark}
\newtheorem{?}[Th]{Problem}

\newcommand{\Geom}{\mathbf{Geom}}

\begin{document}

\title{Asymptotic Properties of Random Contingency Tables with Uniform Margin}

\author[Da Wu]{Da Wu}

\address{University of Pennsylvania \\ Department of Mathematics \\ David Rittenhouse Lab \\ 209 South 33rd Street \\ Philadelphia, PA, 19104-6395} 

\email{dawu@math.upenn.edu}



\begin{abstract} 
Let $C\geq 2$ be a positive integer. Consider the set of $n\times n$ non-negative integer matrices whose row sums and column sums are all equal to $Cn$ and let $X=(X_{ij})_{1\leq i,j\leq n}$ be uniformly distributed on this set. This $X$ is called the random contingency table with uniform margin. In this paper, we study various asymptotic properties of $X=(X_{ij})_{1\leq i,j\leq n}$ as $n\to\infty$.  
\end{abstract} 

\maketitle
\section{Introduction.}
Contingency tables model the dependence structure in large data sets. Mathematically, it is the set of matrices with fixed row and column sum. Let $\mathbf r=(r_1,\ldots,r_m)$ and $\mathbf c=(c_1,\ldots, c_n)$ be two positive integer vectors with same total sum of entries $N$, i.e., 
\begin{equation*}
	\sum_{i=1}^m r_i=\sum_{j=1}^n c_j=N.
\end{equation*}   
The $\mathbf r$ and $\mathbf c$ are called row margin and column margin respectively. Let $\mathscr M(\mathbf r,\mathbf c)$ be the set of $m\times n$ non-negative integer matrices with $i$th row sum $r_i$ and $j$th column sum $c_j$ for $1\leq i\leq m$ and $1\leq j\leq n$, i.e.,
\begin{equation*}
	\mathscr M(\mathbf r,\mathbf c)=\left\lbrace (x_{ij})\in \mathbb Z_{\geq 0}^{mn}: \sum_{j=1}^n x_{kj}=r_k, \sum_{i=1}^m x_{ir}=c_r, \forall \ 1\leq k\leq m, 1\leq r\leq n \right\rbrace.
\end{equation*}
The \textit{Random Contingency Table} $X=(X_{ij})$ is defined as the uniform sample from $\mathscr M(\mathbf r,\mathbf c)$ and we are interested in various asymptotic statistics of $X$ when the dimensions grow to infinity. 
\subsection{Notations.}
\begin{enumerate}
	\item We use $f(n)=O(g(n))$ or $f(n)\ll g(n)$ to denote the estimate $|f(n)|\leq C g(n)$ for some $C$ independent of $n$ and for all $n\geq C$. If this $C$ depends on some parameter $k$, then we will write $f(n)=O_k(g(n))$ or $f(n)\ll_k g(n)$.
	\item Let $\Geom(C)$ denote the geometric distribution with mean $C$. That is, if $X'\sim\Geom(C)$, then 
\begin{equation*}
	\mathbb P(X'=x)=\left(\frac{1}{1+C}\right)\left(\frac{C}{1+C}\right)^x
\end{equation*}
for $x\geq 0$.
\item For every two probability distributions $\mu_1$ and $\mu_2$ on the countable sample space, the \textit{total variation distance metric} is defined as 
\begin{equation*}
	\|\mu_1,\mu_2\|_{TV}:= \frac{1}{2}\sum_{x\in \Omega}|\mu_1(x)-\mu_2(x)|.
\end{equation*}
\item For any measurable set $A$, let $\mathbb E[X;A]$ denote the expectation of $X$ over $A$, i.e.,
\begin{equation*}
	\mathbb E[X;A]=\mathbb E[X\mathbb{1}_A] = \int_A X d\mathbb P.
\end{equation*}
\end{enumerate}
\subsection{Setup and Statements of the Main Results.}
 In this paper, we consider the discrete \textit{random contingency table with uniform margin}. Namely, all the row sums and column sums are equal. Precisely, let $C\geq 2$ be a positive integer and let 
 \begin{equation*}
 	\widetilde{\mathbf r}=\widetilde{\mathbf c}=\left(Cn,\ldots, Cn\right)\in \mathbb Z_{>0}^n,
 \end{equation*}
 where $n$ is the dimension of contingency tables. Denote the set of $n\times n$ non-negative integer matrices whose row rums and column sums are all equal to $Cn$ by $\mathscr M(Cn,n)$.
 Throughout this paper, the matrix-valued random variable $X=(X_{ij})_{1\leq i,j\leq n}$ is uniformly distributed on $\mathscr M(Cn,n)$, i.e.,
 \begin{equation*}
 	\mathbb P(X=D)=\frac{1}{\# \mathscr M(Cn,n)}, \qquad\forall\  D\in \mathscr M(Cn,n).
 \end{equation*} 
 We are interested in various asymptotic statistics of $X$ as $n\to \infty$. By symmetry, all the $X_{ij},1\leq i,j \leq n$, have the same marginal distribution. It was proved in \cite[Theorem $2.1$]{DimtterLyuPak} that the marginal distribution of a single entry of $X$ will converge to $\Geom(C)$ in total variation distance and this can be viewed as the discrete counterpart of \cite[Theorem $1$]{CDS}.
  \begin{Th}[Marginal Distribution, {\cite[Theorem $2.1$]{DimtterLyuPak}}]\label{Marginal Distribution}
 	Let $X=(X_{ij})_{1\leq i,j\leq n}$ be uniformly distributed on $\mathscr M(Cn,n)$, then 
 	\begin{equation*}
 		X_{11}\Longrightarrow \Geom(C)\qquad \text{as}\ n\to \infty,
 	\end{equation*}
 	where convergence is in total variation distance. Moreover, for any $\varepsilon>0$,
 	\begin{equation*}
 		\|X_{11}, \Geom(C)\|_{TV}=O(n^{-1/2+\varepsilon}).
 	\end{equation*}
 \end{Th}
 The proof of this theorem depends heavily on the \textit{Maximum Entropy Principle}, which was introduced by I. J. Good in \cite{Good}. More recently, Alexander Barvinok managed to apply this principle to the context of random contingency tables and answer the question \textit{What does a random contingency table looks like?} He found that as dimensions of the matrix grow, the random contingency table behaves much like the matrix of independent geometric random variables. Precise meaning of \textit{much like} was given in his sequence of papers \cite{Barvinok, Barvinok2, Barvinok3} and we will review these material carefully in the later section. 
 \par
 Going one step further, it is shown in \cite[Theorem $A.3$]{DimtterLyuPak} that if we take a $k\times k$ sub-matrix of $X$, where $k=o\left(\frac{n^{1/2}}{(\log n)^{1/2}} \right)$, then the joint distribution of this sub-matrix converges in total variation distance to the $k\times k$ matrix of independent geometric random variables with mean $C$. This can be viewed as the discrete counterpart of \cite[Theorem $4$]{CDS}.
 \begin{Th}[Joint Distribution, {\cite[Theorem $A.3$]{DimtterLyuPak}}]\label{Joint Distribution}
 	Let $W_k$ denote the projection of a uniform sample $X$ onto the $k\times k$ sub-matrix of its first $k$ rows and columns and let $Y_k$ be the $k\times k$ matrix of independent geometric random variables with mean $C$. When $k=o\left(\frac{n^{1/2}}{(\log n)^{1/2}} \right)$, we have
	\begin{equation*}
		\|W_k,Y_k\|_{TV}=o(1).
	\end{equation*}
 \end{Th}
 \begin{Rem}
 	Notice that the above theorem holds in a greater generality in the sense that as long as we are considering $k^2 = o\left(\frac{n}{\log n}\right)$ entires in $X$, they are asymptotically independent geometric random variables with mean $C$. In other words, the shape does not matter. 
 \end{Rem}  
For the convenience of readers, we will provide self-contained proofs of both Theorem \ref{Marginal Distribution} and Theorem \ref{Joint Distribution} in Section $2$. Next, we show that the moments of entries of $X$ converge to the moments of the i.i.d. $\Geom(C)$ variables.
\begin{Th}[Moment Convergence]\label{Moments Convergence Th}
	Let $(i_1,j_1),\ldots,(i_L,j_L)$ be a fixed sequence of indices and let $\alpha_1,\ldots,\alpha_L$ be $L$ fixed positive integers. Let $X=(X_{ij})_{1\leq i,j\leq n}$ be uniformly distributed on $\mathscr M(Cn,n)$, then 
	\begin{equation*}
		\mathbb E\left[\prod_{k=1}^L X_{i_k,j_k}^{\alpha_k}\right]\to \mathbb E \left[\prod_{k=1}^L Y_{k}^{\alpha_k}\right],
	\end{equation*}
	where $Y_1,\ldots, Y_k$ are i.i.d. $\Geom(C)$. 
\end{Th}
Furthermore, we prove that the maximum entries of $X$ are of the same order as that of $n^2$ i.i.d. $\Geom(C)$ variables. 
\begin{Th}[Maximum Entry]\label{Maximum Entry Th}
	For any fixed $\varepsilon>0$, 
	\begin{align*}
		\lim_{n\to \infty}\mathbb P\left(\max_{1\leq i,j\leq n}X_{ij}>\frac{1}{\log\left(\frac{C+1}{C}\right)}\log\left(\frac{C}{1+C}\cdot n^{2+\varepsilon}\right) \right)=0.
	\end{align*}
\end{Th}
By the above theorem, the distribution of each individual $X_{ij}$ is compactly supported on $[0, \frac{(2+\varepsilon)\log n}{\log(\frac{C+1}{C})}]$ asymptotically. Thus, by the concentration of measure for large Wishart Matrix (see \cite{Guionnet}) and the Mar\v{c}enko-Pastur distribution (see \cite{Pastur}), we can obtain the following result on the limiting empirical singular value distribution of $X$. The proof of Theorem $\ref{Limiting Empirical Singular Value Distribution}$ will be presented in Section $\ref{Section of ESVD}$.  
\begin{Th}[Limiting Empirical Singular Value Distribution]\label{Limiting Empirical Singular Value Distribution}
	Let $X$ be uniformly distributed on $\mathscr M(Cn,n)$ and let $\widetilde{\Upsilon}=\frac{1}{\sqrt{n}}(X-\mathbb E[X])$. Let $0\leq\sigma_1(\widetilde\Upsilon)\leq\sigma_2(\widetilde\Upsilon)\leq\ldots\leq \sigma_n(\widetilde\Upsilon)$ be singular values of $\widetilde\Upsilon$ and let 
	\begin{equation*}
		\mu^s_n(\widetilde\Upsilon)=\frac{1}{n}\sum_{i=1}^n\delta_{\sigma_i(\widetilde\Upsilon)}
	\end{equation*}
	be the empirical singular value distribution of $\widetilde\Upsilon$. Then 
	\begin{equation*}
		\mu_n^s(\widetilde\Upsilon)\to \frac{\sqrt{4C(1+C)-y^2}}{\pi C(1+C)}\mathbb{1}_{[0,2\sqrt{C(C+1)}]}dy
	\end{equation*}
	weakly in probability. 
\end{Th}
\subsection{Comparison with known literature and open problems.}
Our results can be viewed as discrete counterparts of the work of Chatterjee, Diaconis and Sly \cite{CDS} , where they studied the doubly stochastic matrix, which is the set of non-negative real-valued matrices whose row and column sums are all equal to $1$. It is also worth mentioning that in \cite{DimtterLyuPak}, authors studied the non-uniform margin case and obtained the phase transition regarding some asymptotic statistics. The case of non-uniform binary contingency tables was studied in \cite{DW1}. From a combinatorial perspective, recent works of \cite{DW2} and \cite{LPC} compared sharp asymptotics of $\#\mathscr M(\mathbf r,\mathbf c)$ in non-uniform case with the so-called \textit{independence heuristic} introduced in \cite{GIndep}. It was shown in \cite{DW2} that in binary non-uniform case, the independence heuristic overestimates the $\#\mathscr M(\mathbf r,\mathbf c)$ whereas in \cite{LPC}, it was proved that the independence heuristic underestimates the $\#\mathscr M(\mathbf r,\mathbf c)$. All of the above results are based on the \textit{Maximum entropy principle}, which we will set up carefully in Section $2$.
\par
In non-uniform case of two different margins, due to the loose estimate of $\#\mathscr M(\mathbf r,\mathbf c)$, it is still unknown how to obtain the moment convergence (similar to Theorem \ref{Moments Convergence Th}). This is a prerequisite to prove the central limit theorem for certain rows and columns (see \cite{DimtterLyuPak} for detailed discussions). It is even more interesting to understand what would be the limiting empirical singular value distribution in this case. 
\section{Proof of Theorem \ref{Marginal Distribution} and \ref{Joint Distribution}.}
In this section, we provide self-contained proofs of both Theorem \ref{Marginal Distribution} and \ref{Joint Distribution}. First recall the notion of \textit{Typical Table} introduced by Barvinok in \cite{Barvinok}.
 \begin{Def}[Typical Table, {\cite[Definition 1.2]{Barvinok}}]
 	Fix row margin $\mathbf r=(r_1\ldots, r_m)\in \mathbb Z_{>0}^m$ and column margin $\mathbf c=(c_1,\ldots,c_n)\in \mathbb Z_{>0}^n$ and let
 	\begin{align*}
 		\mathcal P(\mathbf r,\mathbf c)=\left\lbrace (x_{ij})\in \mathbb R_{\geq 0}^{mn}: \sum_{j=1}^n x_{kj}=r_k, \sum_{i=1}^m x_{ir}=c_r, \forall \ 1\leq k\leq m, 1\leq r\leq n \right\rbrace.
 	\end{align*}
 	For each $M=(m_{ij})\in \mathcal P(\mathbf r,\mathbf c)$, define 
 	\begin{equation*}
 		g(M)=\sum_{1\leq i\leq m,1\leq j\leq n}f(m_{ij}),
 	\end{equation*}
 	where $f:[0,\infty)\to [0,\infty)$ is defined by 
 	\begin{align*}
 		f(x)=(x+1)\log(x+1)-x\log x.
 	\end{align*}
 	The typical table $Z\in \mathcal P(\mathbf r,\mathbf c)$ is defined as
 	\begin{equation*}
 		Z :={\arg\max}_{M\in  \mathcal P(\mathbf r,\mathbf c)} g(M).
 	\end{equation*}
 \end{Def}
 \begin{Rem}
 \begin{enumerate}
 	\item $f(x)=(x+1)\log(x+1)-x\log x$ is the Shannon-Gibbs-Boltzmann entropy of $\Geom(x)$.
 	\item The function $g$ in the above definition is strictly concave, hence it achieves the unique maximum on $\mathcal P(\mathbf r,\mathbf c)$. Therefore, the typical table is well-defined.
 \end{enumerate}
 \end{Rem}
 \begin{Lemma}
 	The typical table for $\mathscr M(Cn,n)$ is $C\cdot I_{n}$, where $I_n$ is the $n\times n $ matrix with all entries equal to $1$.
 \end{Lemma}
 \begin{proof}
 	Notice that  
 	\begin{equation*}
 		\nabla g = \left(\log(1+1/m_{ij}) \right)_{1\leq i,j\leq n},
 	\end{equation*}
 	and the transportation polytope for $\mathscr M(Cn,n)$ is defined by the intersections of hyperplanes in $\mathbb R^{n^2}_{\geq 0}$ given by  
 	\begin{align*}
 		& h_{i\bullet} := \left(\sum_{k=1}^m m_{ik} \right)-C = 0\qquad\text{for all $1\leq i \leq n$,}\\
 		& h_{\bullet j} := \left(\sum_{k=1}^m m_{kj} \right)-C = 0\qquad\text{for all $1\leq j \leq n$.}
 	\end{align*}
 	The gradient $\nabla h_{i\bullet}$ is the $n\times n$ matrix $E_{i\bullet}$, which has $1$'s in the $i$-th row and $0$'s elsewhere. Similarly, the gradient $\nabla h_{\bullet j}$ is the $n\times n$ matrix $E_{\bullet j}$, which has $1$'s in the $j$-th column and $0$'s elsewhere.
 	\par 
 	Hence, by multivariate Lagrange's method, when evaluated at typical table $Z=(z_{ij})$, there exist some non-negative constants $\lambda_1,\ldots,\lambda_n$ and $\mu_1,\ldots,\mu_n$ such that 
 	\begin{equation*}
 		\log(1+1/z_{ij}) = \lambda_i + \mu_j\qquad\text{for all $1\leq i,j\leq n$}. 
 	\end{equation*}
 	Since all the row sums and column sums are equal to $Cn$, by symmetry, 
 	\begin{equation*}
 		\log(1+1/z_{ij}) = \lambda + \mu \qquad\text{for all $1\leq i,j\leq n$}. 
 	\end{equation*}
 	Therefore, all the $z_{ij}$'s are equal to $C$. 
 \end{proof}
 \begin{Th}[{\cite[Theorem $1.7, 2.1$]{Barvinok}}]\label{Barvinok key estimate}
 	Fix margins $\mathbf r=(r_1\ldots, r_m)\in \mathbb Z_{>0}^m$ and $\mathbf c=(c_1,\ldots,c_n)\in \mathbb Z_{>0}^n$ with $\sum_{i=1}^m r_i=\sum_{j=1}^n c_j=N$. Let $Z=(z_{ij})$ be the typical table for $\mathscr M(\mathbf r,\mathbf c)$ and $Y=(Y_{ij})$ be the $m\times n$ matrix of independent geometric random variables with $Y_{ij}\sim \Geom(z_{ij})$. Then we have the following:
 	\begin{enumerate}
 		\item There exists some absolute constant $\gamma>0$ such that 
 	\begin{equation*}
 		N^{-\gamma(m+n)}e^{g(Z)}\leq \# \mathscr M(\mathbf r,\mathbf c)\leq e^{g(Z)}.
 	\end{equation*}
 	\item Conditioned on being in $\mathscr M(\mathbf r,\mathbf c)$, the matrix $Y$ is uniformly distributed on $\mathscr M(\mathbf r,\mathbf c)$, i.e.
 	\begin{equation*}
 		\mathbb P(Y=D)=e^{-g(Z)} \qquad\text{for all $D\in \mathscr M(\mathbf r,\mathbf c)$.}
 	\end{equation*}
 	\item Combining $1$ and $2$, we have that 
 	\begin{equation*}
 		\mathbb P(Y\in \mathscr M(\mathbf r,\mathbf c))=e^{-g(Z)}\cdot\#\mathscr M(\mathbf r,\mathbf c)\geq N^{-\gamma(m+n)}.
 	\end{equation*}
 	\end{enumerate}
 \end{Th}
 \begin{Rem}
 	By $2$ and $3$ in the above theorem, for fixed measurable set $\mathcal C\subseteq \mathbb R_{+}^{m\times n}$, we have the following transformation of mass inequality:
 	\begin{align}
 	\begin{split}\label{Transform of mass}
 		\mathbb P(Y\in \mathcal C) &\geq \mathbb P(Y\in \mathcal C|Y\in \mathscr M(\mathbf r,\mathbf c))\cdot\mathbb P(Y\in \mathscr M(\mathbf r,\mathbf c)) \\
 		&=\mathbb P(X\in \mathcal C)\cdot\mathbb P(Y\in \mathscr M(\mathbf r,\mathbf c))\\
 		&\geq \mathbb P(X\in \mathcal C)\cdot N^{-\gamma(m+n)}.
 	\end{split}
 	\end{align}
 \end{Rem}
 Now, we are ready to prove the Theorem \ref{Marginal Distribution}.
 \begin{proof}[Proof of Theorem \ref{Marginal Distribution}]
 	Let $Y=(Y_{ij})$ with $Y_{ij}\sim$ i.i.d. $\Geom(C)$. By Theorem \ref{Barvinok key estimate}, we have
 	\begin{equation*}
 		\mathbb P(Y\in \mathscr M(Cn,n))\geq (Cn^2)^{-\gamma' n}
 	\end{equation*}
 	for some absolute constant $\gamma'$. Fix a measurable set $A\subseteq [0,\infty)$, $\mathbb{1}_{\lbrace Y_{ij}\in A\rbrace}$ is a sequence of i.i.d. random variables. Let $\mathcal F=\lbrace (i,j),1\leq i,j\leq n \rbrace$ with $\#\mathcal F=n^2$. Fix $\varepsilon>0$, and by Hoeffding inequality on $\frac{1}{\#\mathcal F}\sum_{(i,j)\in \mathcal F}\mathbb{1}_{\lbrace Y_{ij}\in A\rbrace}$,
 	\begin{align*}
 		\mathbb P\left(\left|\frac{1}{\#\mathcal F}\sum_{(i,j)\in \mathcal F}\mathbb{1}_{\lbrace Y_{ij}\in A\rbrace}-\mathbb P(Y_{11}\in A) \right|>\frac{1}{2}n^{-\frac{1}{2}+\varepsilon}\right) &\leq \exp\left(-\frac{2\cdot\frac{1}{4}n^{-1+2\varepsilon}}{\#\mathcal F\cdot\left(\frac{1}{\#\mathcal F} \right)^2} \right)\\
 		&=\exp\left(-\frac{1}{2}n^{-1+2\varepsilon}\cdot\#\mathcal F \right).
 	\end{align*}
 	Equivalently, 
 	\begin{align*}
 		\mathbb P\left(\left|\frac{1}{n^2}\sum_{1\leq i,j\leq n}\mathbb{1}_{\lbrace Y_{ij}\in A\rbrace}-\mathbb P(Y_{11}\in A) \right|>\frac{1}{2}n^{-\frac{1}{2}+\varepsilon}\right)\leq \exp\left(-\frac{1}{2}n^{1+2\varepsilon} \right).
 	\end{align*}
 	Next, 
 	\begin{align*}
 		&\left|\mathbb P(X_{11}\in A)-\mathbb P(Y_{11}\in A) \right|\\
 		&=\left|\mathbb E\left[\frac{1}{n^2}\sum_{1\leq i,j\leq n}\mathbb{1}_{\lbrace X_{ij}\in A\rbrace} \right]-\mathbb P(Y_{11}\in A) \right|\\
 		&\leq \mathbb E\left[\left|\frac{1}{n^2}\sum_{1\leq i,j\leq n}\mathbb{1}_{\lbrace X_{ij}\in A\rbrace}-\mathbb P(Y_{11}\in A)\right|\right]\\
 		&\leq \frac{1}{2}n^{-\frac{1}{2}+\varepsilon}\mathbb P\left(\left|\frac{1}{n^2}\sum_{1\leq i,j\leq n}\mathbb{1}_{\lbrace X_{ij}\in A\rbrace}-\mathbb P(Y_{11}\in A)\right|\leq \frac{1}{2}n^{-\frac{1}{2}+\varepsilon}\right)\\
 		&+\mathbb P\left(\left|\frac{1}{n^2}\sum_{1\leq i,j\leq n}\mathbb{1}_{\lbrace X_{ij}\in A\rbrace}-\mathbb P(Y_{11}\in A)\right|>\frac{1}{2}n^{-\frac{1}{2}+\varepsilon}\right)\\
 		&\leq \frac{1}{2}n^{-\frac{1}{2}+\varepsilon}+\mathbb P\left(\left|\frac{1}{n^2}\sum_{1\leq i,j\leq n}\mathbb{1}_{\lbrace X_{ij}\in A\rbrace}-\mathbb P(Y_{11}\in A)\right|>\frac{1}{2}n^{-\frac{1}{2}+\varepsilon}\right).
 	\end{align*}
 	Notice that
 	\begin{align*}
 		\mathbb P &\left(\left|\frac{1}{n^2}\sum_{1\leq i,j\leq n}\mathbb{1}_{\lbrace Y_{ij}\in A\rbrace}-\mathbb P(Y_{11}\in A)\right|>\frac{1}{2}n^{-\frac{1}{2}+\varepsilon}\right)\\
 		&\geq \mathbb P\left(\left|\frac{1}{n^2}\sum_{1\leq i,j\leq n}\mathbb{1}_{\lbrace Y_{ij}\in A\rbrace}-\mathbb P(Y_{11}\in A)\right|>\frac{1}{2}n^{-\frac{1}{2}+\varepsilon}\bigg|Y\in \mathscr M(Cn,n)\right)\\
 		&\times\mathbb P(Y\in \mathscr M(Cn,n))\\
 		&= \mathbb P\left(\left|\frac{1}{n^2}\sum_{1\leq i,j\leq n}\mathbb{1}_{\lbrace X_{ij}\in A\rbrace}-\mathbb P(Y_{11}\in A)\right|>\frac{1}{2}n^{-\frac{1}{2}+\varepsilon}\right)\cdot\mathbb P(Y\in \mathscr M(Cn,n))\\
 		&\geq \mathbb P \left(\left|\frac{1}{n^2}\sum_{1\leq i,j\leq n}\mathbb{1}_{\lbrace X_{ij}\in A\rbrace}-\mathbb P(Y_{11}\in A)\right|>\frac{1}{2}n^{-\frac{1}{2}+\varepsilon}\right)\cdot (Cn^2)^{-\gamma' n}.
 	\end{align*}
 	The above estimate follows from the transformation of mass inequality (\ref{Transform of mass}) . The key observation is that if $Y$ is conditioned on being in $\mathscr M(Cn,n)$, then $Y$ is uniformly distributed and has the same law as $X$. Hence, 
 	\begin{align*}
 		&\left|\mathbb P (X_{11}\in A)-\mathbb P(Y_{11}\in A)\right|\\
 		&\leq \frac{1}{2}n^{-\frac{1}{2}+\varepsilon}+(Cn^2)^{\gamma' n}\cdot \mathbb P \left(\left|\frac{1}{n^2}\sum_{1\leq i,j\leq n}\mathbb{1}_{\lbrace Y_{ij}\in A\rbrace}-\mathbb P(Y_{11}\in A)\right|>\frac{1}{2}n^{-\frac{1}{2}+\varepsilon}\right)\\
 		&\leq \frac{1}{2}n^{-\frac{1}{2}+\varepsilon}+(Cn^2)^{\gamma' n}\exp\left(-\frac{1}{2}n^{1+2\varepsilon} \right).
 	\end{align*}
 	Therefore, for any fixed $\varepsilon$, 
 	\begin{equation*}
 		\|X_{11},Y_{11}\|_{TV}=O(n^{-\frac{1}{2}+\varepsilon}),
 	\end{equation*}
 	which is equivalent to 
 	\begin{equation*}
 		\|X_{11},\Geom(C)\|_{TV}=O(n^{-\frac{1}{2}+\varepsilon}).
    \end{equation*}
    This completes the proof of Theorem \ref{Marginal Distribution}.
 \end{proof}
 Furthermore, we can also study the joint distribution of the sub-matrix of $X$. Again, let $Y=(Y_{ij})$ be the matrix of i.i.d. $\Geom(C)$. Fix some $k=k(n)$ and let 
 \begin{equation*}
 	W^{\ell_1\ell_2}=
 	\begin{pmatrix}
 		Y_{\ell_1 k,\ell_2 k} & Y_{\ell_1 k,\ell_2 k+1} & \ldots & Y_{\ell_1 k, \ell_2 k+(k-1)}\\
 		Y_{\ell_1 k+1,\ell_2 k} & Y_{\ell_1 k+1,\ell_2 k+1} & \ldots & Y_{\ell_1 k+1, \ell_2 k+(k-1)}\\
 		\vdots & \vdots & & \vdots\\
 		Y_{\ell_1 k+(k-1),\ell_2 k} & Y_{\ell_1 k+(k-1),\ell_2 k+1} & \ldots & Y_{\ell_1 k+(k-1), \ell_2 k+(k-1)}
 	\end{pmatrix}.
 \end{equation*}
 Also, let 
 \begin{equation*}
 	V^{\ell_1\ell_2}=
 	\begin{pmatrix}
 		X_{\ell_1 k,\ell_2 k} & X_{\ell_1 k,\ell_2 k+1} & \ldots & X_{\ell_1 k, \ell_2 k+(k-1)}\\
 		X_{\ell_1 k+1,\ell_2 k} & X_{\ell_1 k+1,\ell_2 k+1} & \ldots & X_{\ell_1 k+1, \ell_2 k+(k-1)}\\
 		\vdots & \vdots & & \vdots\\
 		X_{\ell_1 k+(k-1),\ell_2 k} & X_{\ell_1 k+(k-1),\ell_2 k+1} & \ldots & X_{\ell_1 k+(k-1), \ell_2 k+(k-1)}
 	\end{pmatrix}.
 \end{equation*}
 \begin{proof}[Proof of Theorem \ref{Joint Distribution}]
 	Let $A\subseteq \mathbb R^{k^2}$ be a measurable set. By Hoeffding inequality, 
 	\begin{align*}
 		&\mathbb P\left(\left|\frac{1}{\left[\frac{n-1}{k}\right]^2}\sum_{\ell_1=1}^{\left[\frac{n-1}{k}\right]}\sum_{\ell_2=1}^{\left[\frac{n-1}{k}\right]}\mathbb{1}_{\lbrace W^{\ell_1\ell_2}\in A\rbrace}-\mathbb P\left(W^{11}\in A\right)\right|>\frac{1}{2}\varepsilon \right)\\
 		&\leq \exp\left(-\frac{\frac{1}{2}\varepsilon^2}{\left[\frac{n-1}{k}\right]^2\frac{1}{\left[\frac{n-1}{k}\right]^4}}\right)=\exp\left(-\frac{1}{2}\varepsilon^2\cdot \left[\frac{n-1}{k}\right]^2 \right).
 	\end{align*}
 	Again, by $(\ref{Transform of mass})$, 
 	\begin{align*}
 		&\mathbb P\left(\left|\frac{1}{\left[\frac{n-1}{k}\right]^2}\sum_{\ell_1=1}^{\left[\frac{n-1}{k}\right]} \sum_{\ell_=1}^{\left[\frac{n-1}{k}\right]}\mathbb{1}_{\lbrace V^{\ell_1\ell_2}\in A\rbrace}-\mathbb P\left(W^{11}\in A\right)\right|>\frac{1}{2}\varepsilon \right)\\
 		&\leq (Cn^2)^{\gamma' n}\cdot \exp\left(-\frac{1}{2}\varepsilon^2\cdot \left[\frac{n-1}{k}\right]^2 \right).
 	\end{align*}
 	By exchangeability of entries of $X$, 
 	\begin{align*}
 		|\mathbb P(V^{11}\in A)-\mathbb P(W^{11}\in A)| &=\left|\mathbb E\left[\frac{1}{\left[\frac{n-1}{k}\right]^2}\sum_{\ell_1=1}^{\left[\frac{n-1}{k}\right]}\sum_{\ell_2=1}^{\left[\frac{n-1}{k}\right]}\mathbb{1}_{\lbrace V^{\ell_1\ell_2}\in A\rbrace}-\mathbb P\left(W^{11}\in A\right) \right] \right|\\
 		&\leq \mathbb E\left[\left|\frac{1}{\left[\frac{n-1}{k}\right]^2}\sum_{\ell_1=1}^{\left[\frac{n-1}{k}\right]}\sum_{\ell_2=1}^{\left[\frac{n-1}{k}\right]}\mathbb{1}_{\lbrace V^{\ell_1\ell_2}\in A\rbrace}-\mathbb P\left(W^{11}\in A\right) \right|\right]\\
 		&\leq \frac{1}{2}\varepsilon+(Cn^2)^{\gamma' n}\cdot\exp\left(-\frac{1}{2}\varepsilon^2\cdot \left[\frac{n-1}{k}\right]^2 \right).
 	\end{align*}
 	Hence, when $k=o\left(\frac{n^{\frac{1}{2}}}{(\log n)^{\frac{1}{2}}} \right)$, 
 	\begin{equation*}
 		\left|\mathbb P(V^{11}\in A)-\mathbb P(W^{11}\in A)\right|\leq \frac{1}{2}\varepsilon+o(1).
 	\end{equation*}
 	Taking $\varepsilon\downarrow 0$ and this completes the proof. 
 \end{proof}
 \section{Convergence of Moments.}
 In this section, we show that finite moments of entries of $X$ converge to those of independent $\Geom(C)$ variables. We first show that the uniform margin maximizes the number of matrices among all the possible margins. 
 \begin{Lemma}\label{Two Column Case}
 	Let $\mathbf a=(a_1,\ldots,a_m)$ be a $m$-dimensional positive integer vector and let $N=a_1+\ldots+a_m$. Let $\mathbf b_r=(r,N-r)$ be a $2$-dimensional positive integer vector. We have that
 	\begin{equation*}
 		\begin{cases}
 			\#\mathscr M(\mathbf a,\mathbf b_r)\leq \#\mathscr M(\mathbf a,\mathbf b_{N/2}) \qquad &\text{if $N\equiv 0\mod 2$,}\\
 			\#\mathscr M(\mathbf a,\mathbf b_r)\leq \#\mathscr M(\mathbf a,\mathbf b_{(N+1)/2})=\#\mathscr M(\mathbf a,\mathbf b_{(N-1)/2}) \qquad &\text{if $N\equiv 1\mod 2$.}
 		\end{cases}
 	\end{equation*} 
 	 \begin{proof}
 		Observe that 
 		\begin{align*}
 			\#\mathscr M(\mathbf a,\mathbf b_r)=N_m(r;a_1,\ldots,a_m),
 		\end{align*}
 		where $N_m(r;a_1,\ldots,a_m)$ denotes the number of ways of partitioning $r$ into $m$ pieces, each of which has size less or equal to $a_i$. Next, consider the generating function of $N_m(r;a_1,\ldots,a_m)$, 
 		\begin{align*}
 			\prod_{\ell=1}^m (1+q+\ldots+q^{a_\ell}) &=\sum_{r=0}^{a_1+\ldots+a_m}N_m(r;a_1,\ldots,a_m) q^r\\
 			&=\sum_{r=0}^{N}N_m(r;a_1,\ldots,a_m) q^r.
 		\end{align*}
 		It suffices to show that 
 		\begin{align}\label{partition order}
 		\begin{split}
 			N_m(r;a_1,\ldots, a_m) &\leq N_m(N/2;a_1,\ldots, a_m)\qquad\qquad\ \ \text{if $N$ even,}\\
 				N_m(r;a_1,\ldots, a_m) &\leq N_m((N+1)/2;a_1,\ldots, a_m)\\
 				&=N_m((N-1)/2;a_1,\ldots, a_m)\qquad\text{if $N$ odd.}
 		\end{split}
 		\end{align}
 		Without loss of generality, assume $a_1\leq a_2\leq \ldots\leq a_m$. We prove $(\ref{partition order})$ by induction. Write 
 		\begin{equation*}
 			\prod_{\ell=1}^{m'}(1+q+q^2+\ldots+q^{a_{\ell}})=c^{m'}_0+c_1^{m'}q+\ldots c_i^{m'}q^i+\ldots+c_{a_1+\ldots+a_{m'}}^{m'} q^{a_1+\ldots+a_{m'}}
 		\end{equation*} 
 		for $1\leq m'\leq m$. Our induction hypotheses are that
 		\begin{enumerate}
 			\item If $\sum_{n=1}^{m'} a_n$ is even, then
 		\begin{align*}
 			1=c^{m'}_0\leq c_1^{m'}\leq \ldots\leq c^{m'}_{\frac{a_1+\ldots+a_{m'}}{2}}\geq c^{m'}_{\frac{a_1+\ldots+a_{m'}}{2}+1}\geq\ldots\geq 1=c^{m'}_{a_1+\ldots+a_{m'}}.
 		\end{align*}
 		\item If $\sum_{n=1}^{m'} a_n$ is odd, then
 		\begin{align*}
 			1=c^{m'}_0\leq c_1^{m'}\leq \ldots\leq c^{m'}_{\frac{a_1+\ldots+a_{m'}+1}{2}}=c^{m'}_{\frac{a_1+\ldots+a_{m'}-1}{2}}\geq\ldots\geq 1=c^{m'}_{a_1+\ldots+a_{m'}}.
 		\end{align*}
 		\item For all $0\leq j\leq a_1+\ldots+a_{m'}$,
 		\begin{align*}
 			c^{m'}_j=c^{m'}_{a_1+\ldots+a_{m'-j}}.
 		\end{align*}
 		\end{enumerate}
  	When $m'=1$, all the coefficients are equal to $1$. It satisfies the hypotheses. Suppose the case of $m'$ is true and we consider the case of $m'+1$. By simply expanding the product, we have that 
 	\begin{align*}
 		&\prod_{\ell=1}^{m'+1}(1+q+\ldots+q^{a_\ell})\\
 		&=\left(1+c_1^{m'}q+\ldots c_i^{m'}q^i+\ldots+c_{a_1+\ldots+a_{m'}}^{m'} q^{a_1+\ldots+a_{m'}}\right)(1+q+\ldots+q^{a_{m'+1}})\\
 		&=1+c_1^{m'}q +\ldots c_i^{m'}q^i+\ldots+c_{a_1+\ldots+a_{m'}}^{m'} q^{a_1+\ldots+a_{m'}}\\
 		&+q+c_1^{m'}q^2 +\ldots c_i^{m'}q^{i+1}+\ldots+c_{a_1+\ldots+a_{m'}}^{m'} q^{a_1+\ldots+a_{m'}+1}\\
 		\vdots\\
 		&+q^{a_{m'+1}}+c_1^{m'}q^{a_{m'+1}+1} +\ldots c_i^{m'}q^{a_{m'+1}+i}+\ldots+c_{a_1+\ldots+a_{m'}}^{m'} q^{a_1+\ldots+a_{m'}+a_{m'+1}}.
 	\end{align*}
 	After summing up coefficients and applying the induction hypothesis, we are done.  
 	\end{proof}
 \end{Lemma}
 \begin{Lemma}\label{Global Volume Case}
 	Let $\mathbf a=(a_1,\ldots,a_m)$ and $\mathbf b=(b_1,\ldots,b_n)$ be two positive integer vectors such that 
 	\begin{equation*}
 		\sum_{i=1}^m a_i = \sum_{j=1}^n b_j = N.
 	\end{equation*} 
 	Suppose $\mathbf b'$ is a vector such that  
 	\begin{equation*}
 		\begin{cases}
 			\mathbf b':=(\frac{b_1+b_2}{2}, \frac{b_1+b_2}{2},b_3,\ldots,b_n) &\qquad\text{if $b_1+b_2$ is even,}\\
 			\mathbf b':=(\frac{b_1+b_2-1}{2}, \frac{b_1+b_2+1}{2},b_3,\ldots,b_n)&\qquad\text{if $b_1+b_2$ is odd,}
 		\end{cases}
 	\end{equation*}
 	then we have that
 	\begin{equation*}
 		\#\mathscr M(\mathbf a,\mathbf b)\leq \#\mathscr M(\mathbf a,\mathbf b').
 	\end{equation*}
 \end{Lemma}
 \begin{proof}
 	Let 
 	\begin{equation*}
 		\mathcal A=\left\lbrace (\widehat a_i)_{i=1,\ldots,m}: \sum_{i=1}^m\widehat a_i=N-b_1-b_2, 0\leq\widehat a_i\leq a_i \right\rbrace
 	\end{equation*}
 	be the set of possible truncated row margins (except the first two columns) and we have 
 	\begin{equation*}
 		\#\mathscr M(\mathbf a,\mathbf b)=\sum_{(\widehat a_i)_{1\leq i\leq m}\in \mathcal A}\# \mathscr M((\widehat a_i)_{1\leq i\leq m}, (b_i)_{i\geq 3})\cdot\#\mathscr M((a_i-\widehat{a_i})_{1\leq i\leq m}, (b_i)_{i=1,2}).
 	\end{equation*}
 	Similarly, 
 	\begin{equation*}
 		\#\mathscr M(\mathbf a,\mathbf b')=\sum_{(\widehat a_i)_{1\leq i\leq m}\in \mathcal A}\# \mathscr M((\widehat a_i)_{1\leq i\leq m}, (b_i)_{i\geq 3})\cdot\# \mathscr M((a_i-\widehat{a_i})_{1\leq i\leq m}, (b_i')_{i=1,2}),
 	\end{equation*}
 	where 
 	\begin{equation*}
 		\begin{cases}
 			b_1'=b_2'=\frac{b_1+b_2}{2} &\qquad\text{if $b_1+b_2$ is even,}\\
 			b_1'=\frac{b_1+b_2-1}{2}, b_2'=\frac{b_1+b_2+1}{2} &\qquad\text{if $b_1+b_2$ is odd.}
 		\end{cases}
 	\end{equation*}
 	By Lemma \ref{Two Column Case},
 	\begin{equation*}
 		\#\mathscr M((a_i-\widehat{a_i})_{1\leq i\leq m}, (b_i)_{i=1,2})\leq \#\mathscr M((a_i-\widehat{a_i})_{1\leq i\leq m}, (b_i')_{i=1,2}),
 	\end{equation*} 
 	which implies that 
 	\begin{equation*}
 		\#\mathscr M(\mathbf a,\mathbf b)\leq \#\mathscr M(\mathbf a,\mathbf b').
 	\end{equation*}
 	This completes the proof.
 \end{proof}
 \begin{Cor}\label{Uniform margin maximize}
 	Fix postive integers $m,n$ and $N$. Let $\ell_1$ and $\ell_2$ be such that
 	\begin{equation*}
 		N\equiv \ell_1\mod m\qquad\text{and}\qquad N\equiv \ell_2\mod n.
 	\end{equation*}
 	 Let  
 	\begin{align*}
 		\mathcal P=\left\lbrace (\mathbf a,\mathbf b)\in \mathbb Z_{>0}^m\times \mathbb Z^n_{>0}:\mathbf a=(a_1,\ldots,a_m),\mathbf b=(b_1,\ldots,b_n), \sum_{i=1}^m a_j=\sum_{j=1}^n b_j=N \right\rbrace
 	\end{align*}
 	be the set of all possible margins and let 
 	\begin{align*}
 		\begin{cases}
 			\mathbf a^*=(\underbrace{\frac{N-\ell_1}{m}+1,\ldots, \frac{N-\ell_1}{m}+1}_{\ell_1\ \text{entries}},\underbrace{\frac{N-\ell_1}{m},\ldots,\frac{N-\ell_1}{m}}_{m-\ell_1 \ \text{entries}}),\\
 			\mathbf b^*=(\underbrace{\frac{N-\ell_2}{n}+1,\ldots, \frac{N-\ell_2}{n}+1}_{\ell_2\ \text{entries}},\underbrace{\frac{N-\ell_2}{n},\ldots,\frac{N-\ell_2}{n}}_{n-\ell_2 \ \text{entries}}).
 		\end{cases}
 	\end{align*}
 	Then we have
 \begin{equation*}
 	\#\mathscr M(\mathbf a^*,\mathbf b^*)=\max_{(\mathbf a,\mathbf b)\in \mathcal P}\#\mathscr{M}(\mathbf a,\mathbf b).
 \end{equation*}
 \end{Cor}
 \begin{proof}
 	Starting with any pair of margins $(\mathbf a,\mathbf b)\in \mathcal P$, applying the Lemma \ref{Global Volume Case} repeatedly on $(\mathbf a,\mathbf b)$ proves the above result. 
 \end{proof}
Next, consider the set $\mathcal R(Cn,r)$, $1\leq r< n$, which consists of all the $r\times n$ non-negative integer matrices with row sums $Cn$. There are no restrictions on column sums. Consider the following two probability measures on $\mathcal R(Cn,r)$:
\begin{enumerate}
	\item $\mu_r$, a uniform measure on $\mathcal R(Cn,r)$.
	\item $\nu_r$, a measure induced by the uniform measure on $\mathscr M(Cn,n)$.
\end{enumerate}
Our next task is to understand the Radon-Nikodym Derivative $\frac{d\nu_r}{d\mu_r}$. Precisely, our goal is to find a uniform finite upper bound on $\frac{d\nu_r}{d\mu_r}$ for all $n$. 
\begin{Lemma}\label{Binom Asymptotics}
	Suppose $0<\gamma<1$ is fixed and $m=(\gamma+o(1))n$. Then 
	\begin{align}
		\binom{n}{m} &=\exp\left[(h(\gamma)+o(1))n \right]\label{Tao Exerises} \\
		&=\exp\left[h(\gamma)n+\frac{1}{2}\log\frac{1}{\gamma}-\frac{1}{2}\log(2\pi(1-\gamma)n) \right]\label{Error Term},
	\end{align}
	where 
	\begin{equation*}
		h(\gamma)=\gamma\log\frac{1}{\gamma}+(1-\gamma)\log\frac{1}{1-\gamma}.
	\end{equation*}
\end{Lemma}
This is the Exercise $1.2.1$ in Tao's book \textit{Topics in Random Matrix Theory} \cite{Tao}. The proof is trivial by the Stirling Formula $n!=\sqrt{2\pi n}\left(\frac{n}{e}\right)^n(1+o(1))$. Here we write down the error term $o(1)$ in $(\ref{Tao Exerises})$ explicitly as in $(\ref{Error Term})$ in order to do a careful analysis later. 
\par
Next, recall an important theorem by Canfield and McKay \cite{CanfieldMcKay} on precise asymptotic enumeration of the number of non-negative integer matrices with equal row and column sums.
\begin{Th}[{\cite[Theorem $1$]{CanfieldMcKay}}]\label{Canfield and McKay}
	Let $\widehat{\mathbf r}=(r,\ldots,r)\in \mathbb Z_{>0}^m$ and $\widehat{\mathbf c}=(c,\ldots,c)\in \mathbb Z_{>0}^n$ with $mr=nc$. Let $\lambda=\frac{r}{n}=\frac{c}{m}$ and let $a,b>0$ be constants such that $a+b<\frac{1}{2}$. Suppose $m, n\to \infty$ in such a way that 
	\begin{equation*}
		\frac{(1+2\lambda)^2}{4\lambda(1+\lambda)}\left(1+\frac{5m}{6n}+\frac{5n}{6m} \right)\leq a\log n,
	\end{equation*}
	then 
	\begin{align*}
		&\# \mathscr M(\widehat{\mathbf r}, \widehat{\mathbf c})\\
		&=\frac{\binom{n+r-1}{r}^m \binom{m+c-1}{c}^n}{\binom{mn+\lambda mn-1}{\lambda mn}}\exp\left(\frac{1}{2}+O(n^{-b}) \right)\\
		&=\frac{(\lambda^{-\lambda}(1+\lambda)^{1+\lambda})^{mn}}{(4\pi A)^{(m+n-1)/2}m^{(n-1)/2}n^{(m-1)/2}}\exp\left(\frac{1}{2}-\frac{1+2A}{24A}\left(\frac{m}{n}+\frac{n}{m}\right)+O(n^{-b}) \right),
	\end{align*}
	where $A=\frac{\lambda(1+\lambda)}{2}$. 
\end{Th} 
\begin{Prop}\label{Estimate on Radon-Nikodym Derivative}
	Fix a positive integer $r$ and for $n>r$, 
	\begin{equation}
		\frac{d\nu_r}{d\mu_r}\leq (1+o(1))e^{\frac{r}{2}}.
	\end{equation}
\end{Prop}
\begin{proof}
	For each $(X_{ij})_{1\leq i,\leq r, 1\leq j\leq n}$, $\nu_r$ gives weights to $(X_{ij})_{1\leq i,\leq r, 1\leq j\leq n}$ proportional to $\# \mathscr M((Cn)_{n-r},(Cn-\sum_{i=1}^r X_{ij})_{j=1,\ldots,n})$. Here, $(Cn)_{n-r}$ just means the $n-r$ dimensional vector with all entries $Cn$. Therefore, 
	\begin{equation*}
		\frac{d\nu_r}{d\mu_r}\left((X_{ij})_{1\leq i\leq r,1\leq j\leq n} \right)\propto  \# \mathscr M((Cn)_{n-r},(Cn-\sum_{i=1}^r X_{ij})_{j=1,\ldots,n}).
	\end{equation*}    
	Next, we determine the constant in the above proportionality. Notice that 
	\begin{align*}
		&\#\mathscr M(Cn,n)\\
		&=\sum_{(X_{ij})_{1\leq i\leq r,1\leq j\leq n}\in \mathcal R(Cn,r)} \# \mathscr M((Cn)_{n-r},(Cn-\sum_{i=1}^r X_{ij})_{j=1,\ldots,n})\\
		&=\# \mathcal R(Cn,r)\int_{\mathcal R(Cn,r)}\# \mathscr M((Cn)_{n-r},(Cn-\sum_{i=1}^r X_{ij})_{j=1,\ldots,n})d\mu_r\left((X_{ij})_{1\leq i\leq r,1\leq j\leq n} \right).
	\end{align*}
	Hence, 
	\begin{align*}
		1&=\nu_r(\mathcal R(Cn,r))\\
		&=\frac{\# \mathcal R(Cn,r)}{\#\mathscr M(Cn,n)}\int_{\mathcal R(Cn,r)}\# \mathscr M((Cn)_{n-r},(Cn-\sum_{i=1}^r X_{ij})_{j=1,\ldots,n})d\mu_r\left((X_{ij})_{1\leq i\leq r,1\leq j\leq n} \right),
	\end{align*}
	and 
	\begin{align}
		\frac{d\nu_r}{d\mu_r}\left((X_{ij})_{1\leq i\leq r,1\leq j\leq n} \right) &=\frac{\# \mathcal R(Cn,r)}{\#\mathscr M(Cn,n)}\cdot \# \mathscr M((Cn)_{n-r},(Cn-\sum_{i=1}^r X_{ij})_{j=1,\ldots,n})\\
		&\leq \frac{\# \mathcal R(Cn,r)}{\#\mathscr M(Cn,n)}\cdot\#\mathscr M\left((Cn)_{n-r},(C(n-r))_n \right). \label{Upper bound for RN}
	\end{align}
	The last inequality above follows from Corollary \ref{Uniform margin maximize}. By Theorem \ref{Canfield and McKay}, 
	\begin{align*}
		\#\mathscr M\left((Cn)_{n-r},(C(n-r))_n \right)=\frac{\binom{n-r+C(n-r)-1}{C(n-r)}^n \binom{n+Cn-1}{Cn}^{n-r}}{\binom{n(n-r)+Cn(n-r)-1}{Cn(n-r)}}\exp\left(\frac{1}{2}+o(1) \right)
	\end{align*}
	and 
	\begin{align*}
		\#\mathscr M(Cn,n)=\frac{\binom{Cn+n-1}{n-1}^n \binom{Cn+n-1}{n-1}^n}{\binom{n^2+C n^2-1}{Cn^2}}\exp\left(\frac{1}{2}+o(1) \right).
	\end{align*}
	Trivially,
	\begin{equation*}
		\# \mathcal R(Cn,r)=\binom{Cn+n-1}{n-1}^r=\exp\left(\left(\frac{\log(1+C)}{1+C}+\frac{C}{1+C}\log\frac{1+C}{C}+o(1)\right)rn \right).
	\end{equation*}
	By plugging them into $(\ref{Upper bound for RN})$ and Lemma \ref{Binom Asymptotics}, we have  
	\begin{align*}
		&\frac{\# \mathcal R(Cn,r)}{\#\mathscr M(Cn,n)}\cdot\#\mathscr M\left((C(n-r))_n, (Cn)_{n-r} \right)\\
		&=\frac{\binom{n-r+C(n-r)-1}{C(n-r)}^n \binom{n+Cn-1}{Cn}^{n}}{\binom{n(n-r)+Cn(n-r)-1}{Cn(n-r)}}\cdot\frac{\binom{n^2+C n^2-1}{Cn^2}}{\binom{Cn+n-1}{n-1}^n \binom{Cn+n-1}{n-1}^n}(1+o(1))\\
		&=\frac{\exp\left[-\frac{n}{2}\log\left(2\pi\cdot\frac{C}{1+C}(n-r+C(n-r)-1) \right)-\frac{1}{2}\log(2\pi\cdot\frac{C}{1+C}(Cn^2+n^2-1) \right]}{\exp\left[-\frac{n}{2}\log(2\pi\cdot\frac{C}{1+C}(Cn+n-1)-\frac{1}{2}\log(2\pi\cdot\frac{C}{1+C}(Cn(n-r)+n(n-r)-1) \right]}\\
		&\times(1+o(1))\\
		&=\frac{\left(2\pi\cdot \frac{C}{1+C}(n-r+C(n-r)-1)\right)^{-\frac{n}{2}}}{\left(2\pi\cdot \frac{C}{1+C}(Cn+n-1)\right)^{-\frac{n}{2}}}\frac{\left(2\pi\cdot\frac{C}{1+C}(Cn^2+n^2-1)\right)^{-\frac{1}{2}}}{\left(2\pi\cdot \frac{C}{1+C}(Cn(n-r)+n(n-r)-1\right)^{-\frac{1}{2}}}\\
		&\times(1+o(1))\\
		&=\frac{\left(Cn(n-r)+n(n-r)-1\right)^{\frac{1}{2}}}{(Cn^2+n^2-1)^{\frac{1}{2}}}\frac{(Cn+n-1)^{\frac{n}{2}}}{\left(C(n-r)+n-r-1\right)^{\frac{n}{2}}}(1+o(1))\\
		&\leq \frac{(Cn+n-1)^{\frac{n}{2}}}{\left(C(n-r)+n-r-1\right)^{\frac{n}{2}}}(1+o(1))\\
		&=\left(1+\frac{Cr+r}{Cn+n-Cr-r-1}\right)^{\frac{n}{2}}(1+o(1))\\
		&\leq e^{\frac{r}{2}}(1+o(1)).
	\end{align*}
	This completes the proof. 
\end{proof}
\begin{Rem}
	From the above proof, we can easily see that
	\begin{equation*}
		\lim_{n\to \infty}\frac{\# \mathcal R(Cn,r)}{\#\mathscr M(Cn,n)}\cdot\#\mathscr M\left((C(n-r))_n, (Cn)_{n-r} \right)=e^{\frac{r}{2}}.
	\end{equation*} 
\end{Rem}
Now, we are ready to prove the convergence in moments for $X$. 
\begin{Th}[Same as Theorem \ref{Moments Convergence Th}]
	Let $(i_1,j_1),\ldots,(i_L,j_L)$ be a fixed sequence of indices and $\alpha_1,\ldots,\alpha_L$ be $L$ fixed positive integers. Let $X=(X_{ij})_{1\leq i,j\leq n}$ be uniformly distributed on $\mathscr M(Cn,n)$, then 
	\begin{equation*}
		\mathbb E\left[\prod_{k=1}^L X_{i_k,j_k}^{\alpha_k}\right]\to \mathbb E \left[\prod_{k=1}^L Y_{k}^{\alpha_k}\right],
	\end{equation*}
	where $Y_1,\ldots, Y_k$ are i.i.d. $\Geom(C)$. 
\end{Th}
\begin{proof}
	By Theorem \ref{Joint Distribution}, the joint distribution of $\left(X_{i_k,j_k}\right)_{1\leq k\leq L}$ converges to i.i.d. $\Geom(C)$ variables in total variation distance. Hence,
	\begin{equation*}
		\mathbb E\left[\prod_{k=1}^L X_{i_k,j_k}^{\alpha_k}; \max_{1\leq k\leq L}X_{i_k,j_k}<M \right]\to \mathbb E \left[\prod_{k=1}^L Y_{k}^{\alpha_k};\max_{1\leq k\leq L}Y_k<M \right].
	\end{equation*}
	Therefore, it suffices to show that 
	\begin{equation*}
		\lim_{M\to \infty}\sup_n \mathbb E\left[\prod_{k=1}^L X_{i_k,j_k}^{\alpha_k}; \max_{1\leq k\leq L}X_{i_k,j_k}\geq M\right]\to 0.
	\end{equation*}
	Without loss of generality, by symmetry, assume that $1\leq i_k,j_k\leq L$ for all $1\leq k\leq L$. Let $\widetilde Y=(\widetilde Y_{ij})$ be uniformly distributed on $\mathcal R(Cn,L)$. Recall that $\mathcal R(Cn,L)$ is the set of $L\times n$ non-negative integer matrices with row sums $Cn$. By Proposition $\ref{Estimate on Radon-Nikodym Derivative}$, 
	\begin{align}\label{3.7}
		\mathbb E\left[\prod_{k=1}^L X_{i_k,j_k}^{\alpha_k}; \max_{1\leq k\leq L}X_{i_k,j_k}\geq M\right]\leq e^{\frac{L}{2}}(1+o(1))\mathbb E\left[\prod_{k=1}^L\widetilde Y_{i_k,j_k}^{\alpha_k}; \max_{1\leq k\leq L}\widetilde Y_{i_k,j_k}\geq M\right].
	\end{align}
	For each fixed $k$, $\widetilde Y_{i_k,j_k}$ has the law of $\mu_1$ on $\mathcal R(Cn,1)$. Hence, 
	\begin{align*}
		\mathbb P(\widetilde Y_{i_k,j_k}=x) &=\frac{\binom{Cn-x+n-2}{n-2}}{\binom{Cn+n-1}{n-1}}\\
		 &=\left(\frac{n-1}{Cn+n-1}\right)\left(\frac{Cn}{Cn+n-2}\cdot\frac{Cn-1}{Cn+n-3}\ldots\frac{Cn-x+1}{Cn+n-x-1} \right).
	\end{align*}
	It is easy to see that 
	\begin{equation*}
		\lim_{n\to \infty}\mathbb P(\widetilde Y_{i_k,j_k}=x)=\left(\frac{1}{1+C} \right)\left(\frac{C}{1+C}\right)^x,
	\end{equation*}
	and for any fixed $\ell<\infty$ and $(i_k,j_k)$, 
	\begin{align*}
		\mathbb E\left[\widetilde{Y}^{\ell}_{i_k,j_k}\right]=\sum_{k=0}^{Cn}\mathbb P(\widetilde Y_{i_k,j_k}=k)k^\ell\xrightarrow{n\to \infty} \sum_{k=0}^{\infty}\left(\frac{1}{1+C} \right)\left(\frac{C}{1+C}\right)^k k^{\ell}<\infty.
	\end{align*}
	Therefore,
	\begin{equation}\label{uniform finiteness of moments}
		\sup_n \mathbb E\left[Y_{i_k,j_k}^{\ell}\right]<\infty
	\end{equation}
	for any finite $\ell<\infty$. By $(\ref{uniform finiteness of moments})$, 
	\begin{align*}
		\sup_n &\mathbb E\left[\prod_{k=1}^L\widetilde Y_{i_k,j_k}^{\alpha_k}; \max_{1\leq k\leq L}\widetilde Y_{i_k,j_k}\geq M\right]\\
		 &=\sup_n \mathbb E\left[\prod_{k=1}^L\left( \widetilde Y_{i_k,j_k}^{\sum_{k=1}^L \alpha_k}\right)^{\frac{\alpha_k}{\sum_{k=1}^L \alpha_k}} ; \max_{1\leq k\leq L}\widetilde Y_{i_k,j_k}\geq M\right]\\
		&\leq \sup_n\mathbb E\left[\sum_{k=1}^L \frac{\alpha_k}{\sum_{k=1}^L\alpha_k}\left(\widetilde{Y}_{i_k,j_k} \right)^{\sum_{k=1}^L \alpha_k}; \max_{1\leq k\leq L}\widetilde Y_{i_k,j_k}\geq M\right]\\
		&\leq \sup_n\mathbb E\left[\sum_{k=1}^L \frac{\alpha_k}{\sum_{k=1}^L\alpha_k}\left(\max_{1\leq k\leq L} \widetilde{Y}_{i_k,j_k} \right)^{\sum_{k=1}^L \alpha_k}; \max_{1\leq k\leq L}\widetilde Y_{i_k,j_k}\geq M\right]\\
		&\leq\frac{1}{M}\sup_n\mathbb E\left[\sum_{k=1}^L \frac{\alpha_k}{\sum_{k=1}^L\alpha_k}\left(\max_{1\leq k\leq L} \widetilde{Y}_{i_k,j_k} \right)^{\sum_{k=1}^L \alpha_k+1}\right]\\
		&\leq \frac{1}{M}\sum_{k=1}^L\frac{\alpha_k}{\sum_{k=1}^L\alpha_k}\sup_n\mathbb E\left[\sum_{k'=1}^L\widetilde Y_{i_{k'},j_{k'}}^{\sum_{k=1}^L \alpha_k+1} \right]\\
		&\to 0
	\end{align*}
	as $M\to \infty$. In the above estimate, the first inequality follows from the power mean inequality $\prod_{i=1}^n x_i^{\omega_i}\leq \sum_{i=1}^n \omega_i x_i$, where $x_i>0$ and $\sum_{i=1}^n\omega_i=1$. The third inequality follows from the following classical estimate: 
	\begin{align*}
		\mathbb E\left[|Y|^{\alpha};Y>M\right]\leq \mathbb E\left[|Y|^{\alpha};|Y|>M\right]\leq \mathbb E\left[\frac{|Y|^{\alpha+1}}{M};|Y|>M\right]\leq \mathbb E\left[\frac{|Y|^{\alpha+1}}{M} \right].
	\end{align*}
	Together with $(\ref{3.7})$, we obtain the desired result. 
\end{proof}
\begin{Cor}
	Let $X=(X_{ij})_{1\leq i,j\leq n}$ be uniformly distributed on $\mathscr M(Cn,n)$, then 
	\begin{equation*}
		\frac{1}{\sqrt{n}}\left(\sum_{j=2}^n X_{1j}-C(n-1)\right)\to 0
	\end{equation*}
	almost surely. 
\end{Cor}
\begin{proof}
	Let $\overline{X_{ij}}=X_{ij}-C$ and 
	\begin{align*}
		\overline{S_{n-1}}=\sum_{j=2}^n \overline{X_{1j}}=\sum_{j=2}^n X_{1j}-C(n-1)=Cn-X_{11}-C(n-1)=C-X_{11}.
	\end{align*}
	For any fixed $\varepsilon>0$, 
	\begin{align*}
		\mathbb P\left(\frac{|\overline{S_{n-1}}|}{\sqrt{n}}>\varepsilon \right) &=\mathbb P\left(|\overline{S_{n-1}}|>\sqrt{n}\varepsilon \right)\\
		&\leq \frac{\mathbb E\left[|\overline{S_{n-1}}|^4 \right]}{\varepsilon^4 n^2}=\frac{\mathbb E\left[|C-X_{11}|^4 \right]}{\varepsilon^4 n^2}=O\left(\frac{1}{n^2} \right).
	\end{align*}
	By Borel-Cantelli Lemma, 
	\begin{align*}
		\mathbb P\left(\frac{|\overline{S_{n-1}}|}{\sqrt{n}}>\varepsilon\ \  \text{i.o.} \right)=0,
	\end{align*}
	and 
	\begin{equation*}
		\frac{1}{\sqrt{n}}\left(\sum_{j=2}^n X_{1j}-C(n-1)\right)\to 0
	\end{equation*}
	almost surely. 
\end{proof}
\section{Maximum Entry.}
 By the maximum entropy principle, we believe that the uniform distributed $X=(X_{ij})$ behaves in many sense like the matrix $Y$ of i.i.d. $\Geom(C)$ variables. In this section, we show that the maximum entries of $X$ and $Y$ are of the same order.  Let's first take a look at the behaviors of the maximum of i.i.d. $\Geom(C)$. This question has been well studied. See \cite{Maxofgeometric} for a more detailed treatment. By the exact same approach there, we have the following Lemma:
\begin{Lemma}
   Let $Y_i$, $1\leq i\leq n^2$ be sequence of i.i.d. $\Geom(C)$, then 
   \begin{enumerate}
   	\item For any $\varepsilon>0$,
   \begin{equation*}
   	\lim_{n\to \infty}\mathbb P\left(\max_{1\leq i\leq n^2}Y_i\leq \frac{1}{\log\left(\frac{C+1}{C}\right)}\log\left(\frac{C}{1+C}\cdot n^{2+\varepsilon}\right) \right)=1.
   \end{equation*}
   \item We have that
   \begin{equation*}
   	\frac{1}{\log\left(\frac{1+C}{C}\right)}H_{n^2}-1<\mathbb E\left[\max_{1\leq i\leq n^2}Y_i\right]<\frac{1}{\log\left(\frac{1+C}{C}\right)}H_{n^2},
   \end{equation*}
   where $H_{n^2}=\sum_{k=1}^{n^2}\frac{1}{k}=\log\left(n^2\right)+\gamma+O(1/n^2)$.
   \end{enumerate}
\end{Lemma}
\begin{proof}
	First, we prove $1$. Since $Y_1\sim \Geom(C)$, the probability density function $\mathbb P(Y_1=x)=\left(\frac{1}{1+C}\right)\left(\frac{C}{1+C}\right)^x$ and the cumulative density function takes the following form:
	\begin{equation*}
     \mathbb P(Y_1\leq x)=\sum_{k=0}^x \mathbb P(Y_1=x)=\sum_{k=0}^x \left(\frac{1}{1+C}\right)\left(\frac{C}{1+C}\right)^x=1-\left(\frac{C}{1+C}\right)^{x+1}.
     \end{equation*}
     Hence, 
     \begin{equation*}
     	\mathbb P\left(\max_{1\leq i\leq n^2}Y_i\leq x\right)=\left[1-\left(\frac{C}{1+C}\right)^{x+1}\right]^{n^2}, 
     \end{equation*}
     and it is easy to verify that
     \begin{equation*}
     	\lim_{n\to \infty}\mathbb P\left(\max_{1\leq i\leq n^2}Y_i\leq \frac{1}{\log\left(\frac{C+1}{C}\right)}\log\left(\frac{C}{1+C}\cdot n^{2+\varepsilon}\right)\right) = 1.
     \end{equation*}
     Next, we prove $2$. To start, we have
     \begin{align*}
     	\mathbb E\left[\max_{1\leq i\leq n^2}Y_i\right]=\sum_{x=0}^{\infty}\mathbb P\left(\max_{1\leq i\leq n^2}Y_i>x \right)=\sum_{x=0}^{\infty}1-\left[1-\left(\frac{C}{1+C}\right)^{x+1}\right]^{n^2}.
     \end{align*}
     Notice that 
     \begin{align*}
     	\sum_{x=0}^{\infty}1-\left[1-\left(\frac{C}{1+C}\right)^{x+1}\right]^{n^2}=\sum_{x=0}^{\infty}1-\left[1-\left(\frac{C}{1+C}\right)^{x}\right]^{n^2}-1.
     \end{align*}
     Let $e^{-\lambda}=\frac{C}{1+C}$, i.e., $\lambda=\log\left(\frac{1+C}{C} \right)$, then 
     \begin{align*}
     	\int_{0}^{\infty}1-(1-e^{-\lambda x})^{n^2}dx<\sum_{x=0}^{\infty}1-\left[1-\left(\frac{C}{1+C}\right)^{x}\right]^{n^2}<1+\int_{0}^{\infty}1-(1-e^{-\lambda x})^{n^2}dx.
     \end{align*}
     Equivalently, 
     \begin{equation*}
     	\frac{1}{\lambda}H_{n^2}<\sum_{x=0}^{\infty}1-\left[1-\left(\frac{C}{1+C}\right)^{x}\right]^{n^2}<1+\frac{1}{\lambda}H_{n^2}.
     \end{equation*}
     This completes the proof of $2$. 
\end{proof}
It turns out that we have a similar result for $X$. 
\begin{Th}[Same as Theorem \ref{Maximum Entry Th}]
	Fix any fixed $\varepsilon>0$, 
	\begin{align*}
		\lim_{n\to \infty}\mathbb P\left(\max_{1\leq i,j\leq n}X_{ij}>\frac{1}{\log\left(\frac{C+1}{C}\right)}\log\left(\frac{C}{1+C}\cdot n^{2+\varepsilon}\right) \right)=0.
	\end{align*}
\end{Th}
\begin{proof}
	By Proposition $\ref{Estimate on Radon-Nikodym Derivative}$,
	\begin{align}\label{Thm 4.2 inequality}
	\begin{split}
		&\mathbb P\left(X_{11}>\frac{1}{\log\left(\frac{C+1}{C}\right)}\log\left(\frac{C}{1+C}\cdot n^{2+\varepsilon}\right) \right)\\
		&\leq (e^{1/2}+o(1))\mathbb P\left(\widetilde Y_{11}>\frac{1}{\log\left(\frac{C+1}{C}\right)}\log\left(\frac{C}{1+C}\cdot n^{2+\varepsilon}\right) \right).
	\end{split}
	\end{align}
	Recall that $\widetilde Y_{11}$ is uniformly distributed on $\mathcal R(Cn,1)$ and has the following probability density function: 
	\begin{align*}
		\mathbb P \left(\widetilde Y_{11}=x \right) &=\frac{\binom{Cn-x+n-2}{n-2}}{\binom{Cn+n-1}{n-1}}\\
		&=\left(\frac{n-1}{Cn+n-1}\right)\left(\frac{Cn}{Cn+n-2}\cdot\frac{Cn-1}{Cn+n-3}\ldots\frac{Cn-x+1}{Cn+n-x-1} \right).
	\end{align*}
	 Notice that 
	\begin{align*}
		\mathbb P &\left(\widetilde Y_{11}>\frac{1}{\log\left(\frac{C+1}{C}\right)}\log\left(\frac{C}{1+C}\cdot n^{2+\varepsilon}\right)\right)\\
		&=\sum_{x=\left[\frac{1}{\log\left(\frac{C+1}{C}\right)}\log\left(\frac{C}{1+C}\cdot n^{2+\varepsilon}\right)\right]+1}^{Cn}\mathbb P\left(\widetilde Y_{11}=x \right)\\
		&=\left(\sum_{\frac{\log(n^{2+\varepsilon})}{\log(\frac{C+1}{C})}\leq x\leq Cn}\mathbb P(\widetilde Y_{11}=x)\right)+o(1).\\
	\end{align*}
	Therefore, 
	\begin{align*}
		&\mathbb P\left(\widetilde Y_{11}>\frac{1}{\log\left(\frac{C+1}{C}\right)}\log\left(\frac{C}{1+C}\cdot n^{2+\varepsilon}\right)\right)\\
		&=\sum_{C_{n,\varepsilon}\leq x\leq Cn}\left(\frac{n-1}{Cn+n-1}\right)\left(\frac{Cn}{Cn+n-2}\cdot\frac{Cn-1}{Cn+n-3}\ldots\frac{Cn-x+1}{Cn+n-x-1}\right)+o(1)\\
		&=\sum_{C_{n,\varepsilon}\leq x\leq Cn}\left(\frac{1}{1+C+\frac{C}{n-1}}\right)\cdot\left( \frac{C}{1+C-\frac{2}{n}}\cdot\frac{C-\frac{1}{n}}{1+C-\frac{3}{n}}\ldots\frac{C-\frac{x-1}{n}}{1+C-\frac{x+1}{n}} \right)+o(1)\\
		&\ll \sum_{C_{n,\varepsilon}\leq x\leq Cn}\left(\frac{1}{1+C}\right)\left(\frac{C}{1+C}\right)^x+o(1)\\
		&\ll\left(\frac{C}{1+C}\right)^{C_{n,\varepsilon}}\left[1-\left(\frac{C}{1+C}\right)^{Cn-C_{n,\varepsilon}}\right]+o(1)\\
		&\ll n^{-2-\varepsilon},
	\end{align*}
	where $C_{n,\varepsilon} = \frac{\log(n^{2+\varepsilon})}{\log\left(\frac{C+1}{C}\right)}$. Hence $\mathbb P\left(\widetilde Y_{11}>\frac{1}{\log\left(\frac{C+1}{C}\right)}\log\left(\frac{C}{1+C}\cdot n^{2+\varepsilon}\right)\right)$ decays to $0$ of order $n^{-2-\varepsilon}$ and by $(\ref{Thm 4.2 inequality})$,
	\begin{equation*}
		\mathbb P\left(X_{11}>\frac{1}{\log\left(\frac{C+1}{C}\right)}\log\left(\frac{C}{1+C}\cdot n^{2+\varepsilon}\right) \right)=O(n^{-2-\varepsilon}).
	\end{equation*}
Finally, by symmetry of $X_{ij}$ and the union bound, we are done. 
\end{proof}
\section{Empirical Singular Value Distribution.}\label{Section of ESVD}
In this section, we study the distribution of singular values of $\frac{1}{\sqrt{n}}(X-\mathbb E[X])$, where $X$ is uniformly distributed on $\mathscr M(Cn,n)$. To begin with, let $0\leq \sigma_1(n)\leq \sigma_2(n)\leq \ldots\leq \sigma_n(n)$ denote the singular values of the matrix $\frac{1}{\sqrt{n}}(X-\mathbb E[X])$. These are the positive square roots of eigenvalues of $\frac{1}{n}(X-\mathbb E[X])(X-\mathbb E[X])^*$. For any $n\times n$ Hermitian matrix $H$, let 
\begin{equation*}
	\mu_n(H)=\frac{1}{n}\sum_{i=1}^n\delta_{\lambda_i(H)}
\end{equation*}
denote the empirical eigenvalue distribution of $H$, where $\lambda_i$ denotes the $i$th eigenvalue of $H$. Let $\widetilde{\mathscr M(Cn,n)}$ be the subset of $\mathscr M(Cn,n)$ with maximum entries not exceeding $\frac{10\log n}{\log\left(\frac{C+1}{C}\right)}$, i.e.,
\begin{align*}
	\widetilde{\mathscr M(Cn,n)}=\left\lbrace \left(\widetilde M_{ij}\right):\sum_{k=1}^n\widetilde{M_{kj}}=Cn,\sum_{k=1}^n\widetilde{M_{ik}}=Cn, \max_{1\leq i,j\leq n}\widetilde{M_{ij}}\leq \frac{10\log n}{\log\left(\frac{C+1}{C}\right)} \right\rbrace.
\end{align*}
Let $\widetilde Y=\left(\widetilde {Y_{ij}}\right)_{1\leq i,j\leq n}$ be the matrix of i.i.d. $\Geom(C)$ variables conditioned on not exceeding $\frac{10\log n}{\log\left(\frac{C+1}{C}\right)}$, i.e.,
\begin{equation*}
	\mathbb P\left(\widetilde {Y_{ij}}=x\right)=
	\begin{cases}
		\frac{1}{Z_n}\left(\frac{1}{1+C}\right)\left(\frac{C}{1+C}\right)^x\qquad  &0\leq x\leq \left[\frac{10\log n }{\log\left(\frac{C+1}{C}\right)}\right]\\
		0 \qquad &\text{otherwise}
	\end{cases},
\end{equation*}
where 
\begin{equation*}
	Z_n=\sum_{x=0}^{\left[\frac{10 \log n}{\log\left(\frac{C+1}{C}\right)}\right]}\left(\frac{1}{1+C}\right)\left(\frac{C}{1+C}\right)^x.
\end{equation*} 
\begin{Lemma}\label{Truncated Lemma}
We have the following:
\begin{enumerate}
	\item Conditioning on $\widetilde Y\in \widetilde{\mathscr M(Cn,n)}$, the $\widetilde Y$ is uniformly distributed on $\widetilde{\mathscr M(Cn,n)}$.
	\item There exists some absolute constant $\gamma^{''}>0$ such that $\mathbb P\left(\widetilde Y\in \widetilde{\mathscr M(Cn,n)}\right)\geq n^{-\gamma^{''} n}$. 
\end{enumerate}
\end{Lemma}
\begin{proof}
	First, we prove $1$. Fix any matrix $D=(D_{ij})\in \widetilde{\mathscr M(Cn,n)}$, 
	\begin{align*}
		\mathbb P\left(\widetilde Y=D\right) &=\prod_{1\leq i,j\leq n}\mathbb P\left(\widetilde{Y_{ij}}=D_{ij} \right)\\
		&=\prod_{1\leq i,j\leq n}\frac{1}{Z_n}\left(\frac{1}{1+C}\right)\left(\frac{C}{1+C}\right)^{D_{ij}}\\
		&=\left(\frac{1}{Z_n}\right)^{n^2}\left(\frac{1}{1+C}\right)^{n^2}\left(\frac{C}{1+C}\right)^{\sum_{1\leq i,j\leq n}D_{ij}}\\
		&=\left(\frac{1}{Z_n}\right)^{n^2}\left(\frac{1}{1+C}\right)^{n^2}\left(\frac{C}{1+C}\right)^{Cn^2}.
	\end{align*}
	Hence, the probability density function of $\widetilde Y$ is constant on $\widetilde{\mathscr M(Cn,n)}$. This proves $1$. Next, we prove $2$. Observe that 
	\begin{align*}
		Z_n &=\sum_{x=0}^{\left[\frac{10}{\log\left(\frac{C+1}{C}\right)}\log n\right]}\left(\frac{1}{1+C}\right)\left(\frac{C}{1+C}\right)^x\\
		&=\frac{1}{1+C}\cdot\frac{1-\left(\frac{C}{1+C}\right)^{\left[\frac{10}{\log\left(\frac{C+1}{C}\right)}\log n\right]+1}}{1-\frac{C}{1+C}}\\
		&=1-\left(\frac{C}{1+C}\right)^{\left[\frac{10}{\log\left(\frac{C+1}{C}\right)}\log n\right]+1}\\
		&=1-\left(\frac{C}{1+C}\right)^{\frac{10}{\log\left(\frac{C+1}{C}\right)}\log n}\cdot\left(\frac{C}{1+C}\right)^{\left[\frac{10}{\log\left(\frac{C+1}{C}\right)}\log n\right]+1-\frac{10}{\log\left(\frac{C+1}{C}\right)}\log n}\\
		&=1-n^{-10}\cdot C_n',
	\end{align*}
	where $C_n'=\left(\frac{C}{1+C}\right)^{\left[\frac{10}{\log\left(\frac{C+1}{C}\right)}\log n\right]+1-\frac{10}{\log\left(\frac{C+1}{C}\right)}\log n}$ and it is easy to see that $\frac{C}{1+C}\leq C_n'\leq 1$. Next, 
	\begin{align}\label{5.1}
	\begin{split}
		&\mathbb P\left(\widetilde Y\in \widetilde{\mathscr M(Cn,n)}\right)\\
		&=\# \widetilde{\mathscr M(Cn,n)}\cdot \left(\frac{1}{Z_n}\right)^{n^2}\left(\frac{1}{1+C}\right)^{n^2}\left(\frac{C}{1+C}\right)^{Cn^2}\\
		&=\# \widetilde{\mathscr M(Cn,n)}\cdot (1-C_n'\cdot n^{-10})^{-n^2}\left(\frac{1}{1+C}\right)^{n^2}\left(\frac{C}{1+C}\right)^{Cn^2}\\
		&=(1+o(1))\cdot \# \widetilde{\mathscr M(Cn,n)}\cdot \left(\frac{1}{1+C}\right)^{n^2}\left(\frac{C}{1+C}\right)^{Cn^2}.
	\end{split}
	\end{align}
	By Theorem \ref{Maximum Entry Th}, 
	\begin{align}\label{5.2}
		\frac{\# \widetilde{\mathscr M(Cn,n)}}{\#\mathscr M(Cn,n)}=\mathbb P\left(\max_{1\leq i,j\leq n}X_{ij}\leq \frac{10}{\log\left(\frac{C+1}{C}\right)}\log n \right)=1-o(1).
	\end{align}
	Combining $(\ref{5.1})$, $(\ref{5.2})$, and Theorem $\ref{Canfield and McKay}$ by Canfield and McKay, we have that 
	\begin{align*}
		&\mathbb P\left(\widetilde Y\in \widetilde{\mathscr M(Cn,n)}\right)\\
		&=(1+o(1))\cdot\left(\frac{1}{1+C}\right)^{n^2}\left(\frac{C}{1+C}\right)^{Cn^2}\frac{(C^{-C}(1+C)^{1+C})^{n^2}}{(2\pi C(1+C))^{\frac{2n-1}{2}}n^{n-1}}\cdot\exp(O(1))\\
		&=(1+o(1))\cdot\exp(O(1))\cdot\frac{1}{(2\pi C(1+C))^{n-\frac{1}{2}}n^{n-1}}\\
		&\geq n^{-\gamma^{''}n},
	\end{align*}
	where $\gamma^{''}>0$ is some absolute constant. 
\end{proof}
Let $\Upsilon=\frac{1}{n}(\widetilde Y-\mathbb E[\widetilde Y])(\widetilde Y-\mathbb E[\widetilde Y])^*$. By results of Guionnet and Zeitouni \cite[Corollary $1.8$]{Guionnet} on the concentration of measure for empirical spectrum of Wishart Matrix, 
\begin{align}\label{Guionnet and Zeitouni}
	\mathbb P\left(W_1\left(\mu_n\left(\Upsilon\right),\mathbb E[\mu_n(\Upsilon)]\right)>\varepsilon \right)\leq \exp(-C'(\epsilon)\cdot n^2\cdot(\log n)^{-2})
\end{align}
for some constant $C'(\epsilon)>1$ and large $n$. Here $W_1$ is the Wasserstein distance defined by
\begin{equation*}
	W_1(\mu,\nu):=\inf_{\gamma\in \Gamma(\mu,\nu)}\left(\int_{\mathbb R \times \mathbb R}|x-y|\gamma(dx,dy)\right),
\end{equation*}
where $\Gamma(\mu,\nu)$ is the set of coupling of $\mu$ and $\nu$ on $\mathbb R\times \mathbb R$.
In addition, by the work of Mar\v {c}enko and Pastur \cite{Pastur} , 
\begin{equation*}
	\mu_n(\Upsilon)\to \mu'
\end{equation*}
weakly in probability, where 
\begin{equation*}
	\mu'=\frac{\sqrt{[4C(1+C)-x]x}}{2\pi C(1+C)}\mathbb{1}_{[0,4C(1+C)]}dx.
\end{equation*}
 Consequently, $\mathbb E[\mu_n(\Upsilon)]\to \mu'$ weakly. Finally, we are ready to prove the following Theorem \ref{Limiting Empirical Singular Value Distribution}: 
\begin{Th}[Same as Theorem \ref{Limiting Empirical Singular Value Distribution}]
	Let $X$ be uniformly distributed on $\mathscr M(Cn,n)$ and let $\widetilde{\Upsilon}=\frac{1}{\sqrt{n}}(X-\mathbb E[X])$. Let $0\leq\sigma_1(\widetilde\Upsilon)\leq\sigma_2(\widetilde\Upsilon)\leq\ldots\leq \sigma_n(\widetilde\Upsilon)$ be singular values of $\widetilde\Upsilon$ and let
	\begin{equation*}
		\mu^s_n(\widetilde\Upsilon)=\frac{1}{n}\sum_{i=1}^n\delta_{\sigma_i(\widetilde\Upsilon)}
	\end{equation*}
	be the empirical singular value distribution of $\widetilde\Upsilon$. Then 
	\begin{equation*}
		\mu_n^s(\widetilde\Upsilon)\to \frac{\sqrt{4C(1+C)-y^2}}{\pi C(1+C)}\mathbb{1}_{[0,2\sqrt{C(C+1)}]}dy
	\end{equation*}
	weakly in probability. 
\end{Th}
\begin{proof}
	By Lemma \ref{Truncated Lemma},
	\begin{align*}
		&\mathbb P(W_1(\mu_n(\Upsilon),\mathbb E(\mu_n(\Upsilon))>\varepsilon)\\
		&\geq \mathbb P (W_1(\mu_n(\Upsilon),\mathbb E(\mu_n(\Upsilon))>\varepsilon|\widetilde Y\in \widetilde{\mathscr M(Cn,n)})\cdot\mathbb P(\widetilde Y\in {\widetilde{\mathscr{M}(Cn,n)}})\\
		&=\mathbb P\left(W_1\left(\mu_n\left(\frac{1}{n}(\widetilde X-\mathbb E[\widetilde X])(\widetilde X-\mathbb E[\widetilde X])^* \right), \mathbb E[\mu_n(\Upsilon)]\right)>\varepsilon\right)\mathbb P(\widetilde Y\in {\widetilde{\mathscr{M}(Cn,n)}}),
	\end{align*}
	where $\widetilde X$ is uniformly distributed on $\widetilde{\mathscr M(Cn,n)}$. Therefore, by (\ref{Guionnet and Zeitouni}),
	\begin{align*}
		&\mathbb P\left(W_1\left(\mu_n\left(\frac{1}{n}(\widetilde X-\mathbb E[\widetilde X])(\widetilde X-\mathbb E[\widetilde X])^* \right), \mathbb E[\mu_n(\Upsilon)]\right)>\varepsilon\right)\\
		&\leq \frac{\mathbb P(W_1(\mu_n(\Upsilon),\mathbb E(\mu_n(\Upsilon))>\varepsilon)}{\mathbb P(\widetilde Y\in {\widetilde{\mathscr{M}(Cn,n)}})}\\
		&\leq n^{\gamma^{''}n}\cdot\exp(-C'(\epsilon)\cdot n^2\cdot(\log n)^{-2})\\
		&=o(1).
	\end{align*}
	Let $E_{C,n}$ be the event that maximum entries of $X$ do not exceed $\frac{10\log n}{\log\left(\frac{C+1}{C}\right)}$, i.e.,
	\begin{equation*}
		E_{C,n} = \left\lbrace X=(X_{ij}):\max_{1\leq i,j\leq n} X_{ij}\leq \frac{10\log n}{\log(\frac{C+1}{C})} \right\rbrace. 
	\end{equation*}
	Since $\widetilde X\sim\left(X|E_{C,n}\right)$, we have
	\begin{align*}
		\mathbb P\left(W_1\left(\mu_n\left(\frac{1}{n}(X-\mathbb E[X])(X-\mathbb E[X])^* \right), \mathbb E[\mu_n(\Upsilon)]\right)>\varepsilon\bigg|E_{C,n} \right)=o(1).
	\end{align*}
	Now, by Theorem $\ref{Maximum Entry Th}$, $\mathbb P(E_{C,n})\to 1$ as $n\to \infty$, which implies that 
	\begin{equation*}
		\mathbb P\left(W_1\left(\mu_n\left(\frac{1}{n}(X-\mathbb E[X])(X-\mathbb E[X])^* \right), \mathbb E[\mu_n(\Upsilon)]\right)>\varepsilon\right)=o(1).
	\end{equation*}
	Since $\mathbb E[\mu_n(\Upsilon)]\to \mu'$, we have 
	\begin{equation*}
		\mu_n\left(\frac{1}{n}(X-\mathbb E[X])(X-\mathbb E[X])^* \right)\to \mu'=\frac{\sqrt{[4C(1+C)-x]x}}{2\pi C(1+C)}\mathbb{1}_{[0,4C(1+C)]}dx
	\end{equation*}
	weakly in probability as $n\to \infty$. By simple change of variables $x=y^2$,  
	\begin{align*}
		\int_a^b \frac{\sqrt{[4C(C+1)-x]x}}{2\pi C(1+C)x}dx=\int_{\sqrt{a}}^{\sqrt{b}}\frac{\sqrt{4C(C+1)-y^2}}{\pi C(1+C)}dy.
	\end{align*}
	Hence,
	\begin{align*}
		\mu^s_n(X)\to \frac{\sqrt{4C(1+C)-y^2}}{\pi C(1+C)}\mathbb{1}_{[0,2\sqrt{C(C+1)}]}dy
	\end{align*}
	weakly in probability. 
\end{proof}

\section*{acknowledgements}
 The author would like to thank Robin Pemantle for many helpful discussions.

\end{document}